\documentclass{amsart}

\usepackage{latexsym}
\usepackage{fullpage}

\usepackage{amsmath}
\usepackage{amsthm}
\usepackage{amssymb}
\usepackage{amsfonts}
\usepackage{here}
\usepackage{enumerate}
\usepackage{mathrsfs}
\usepackage{comment}
\usepackage[pdftex]{graphicx}

\newtheorem{thm}{Theorem}[section]
\newtheorem{lem}[thm]{Lemma}
\newtheorem{crl}[thm]{Corollary}
\newtheorem{cjt}[thm]{Conjecture}
\newtheorem{prop}[thm]{Proposition}
\newtheorem{ex}[thm]{Example}
\newtheorem{fact}[thm]{Fact}
\newtheorem{problem}[thm]{Problem}
\newtheorem{note}[thm]{Notation}

\theoremstyle{definition}
\newtheorem{dfn}[thm]{Definition}
\newtheorem{rem}[thm]{Remark}

\renewcommand{\P}[0]{\mathbb{P}}
\newcommand{\N}[0]{\mathbb{N}}
\newcommand{\Z}[0]{\mathbb{Z}}
\newcommand{\C}[0]{\mathbb{C}}
\newcommand{\pas}[1]{\left(#1\right)}
\newcommand{\Pas}[1]{\left\{#1\right\}}
\newcommand{\floor}[1]{\left\lfloor #1 \right\rfloor}
\newcommand{\abs}[1]{\left\lvert#1\right\rvert}
\newcommand{\flog}[1]{\frac{1}{\log #1}}
\newcommand{\ep}{\varepsilon}

\newcommand{\Falpha}{\mathcal{F}_\alpha}
\newcommand{\fdf}{{}_{\Falpha}\sigma}
\newcommand{\FDF}[1]{{}_{\Falpha}\sigma\pas{#1}}
\newcommand{\FDFk}[2]{{}_{\Falpha}\sigma^{#1}\pas{#2}}
\newcommand{\ord}[1]{\textup{ord}_\alpha\pas{#1}}
\newcommand{\ind}[1]{\textup{ind}_\alpha\pas{#1}}
\newcommand{\valpha}{\varphi_\alpha}

\makeatletter
\def\@seccntformat#1{\csname the#1\endcsname.\quad}
\makeatother

\begin{document}
\title{The relation between a generalized Fibonacci sequence and the length of Cunningham chains}
\author{Yuya Kanado}
\address{Yuya Kanado\\
	Graduate School of Mathematics\\ Nagoya University\\ Furo-cho\\ Chikusa-ku\\ Nagoya\\ 464-8602\\ Japan}
\email{m21017a@math.nagoya-u.ac.jp}
\date{}

\begin{abstract}
	Let $p$ be a prime number. A chain $\{p,2p+1,4p+3,\cdots,(p+1)2^{l(p)-1}-1\}$ is called the Cunningham chain generated by $p$ if all elements are prime numbers and $(p+1)2^{l(p)}-1$ is composite. Then $l(p)$ is called the length of the Cunningham chain. It is conjectured by Bateman and Horn in 1962 that the number of prime numbers $p\leq N$ such that $l(p)\geq k$ is asymptotically equal to $B_k N/(\log N)^k$ with a real $B_k>0$ for all natural numbers $k$. This suggests that $l(p)=\Omega(\log p/\log\log p)$. However, so far no good estimation is known. It has not even been proven whether $\limsup_{p\to\infty} l(p)$ is infinite or not. All we know is that $l(p)=5$ if $p=2$ and $l(p)<p$ for odd $p$ by Fermat's little theorem. Let $\alpha\geq3$ be an integer. In this article, a generalized Fibonacci sequence $\Falpha=\{F_n\}_{n=0}^\infty$ is defined as $F_0=0,F_1=1, F_{n+2}=\alpha F_{n+1}+F_n (n\geq0)$, and $\FDF{n}=\sum_{d\mid n, 0<d\in\Falpha}d$ is called the divisor function on $\Falpha$. Then we obtain an interesting relation between the iteration of $\fdf$ and the length of Cunningham chains. For two prime numbers $p$ and $q$, the fact $p=2q+1$ or $2q-1$ is equivalent to ${}_{\Falpha}\sigma({}_{\Falpha}\sigma(F_p))={}_{\Falpha}\sigma(F_q)$ for some $\alpha$. By this relation, we get $l(p)\ll\log p$ under a certain condition. It seems that this sufficient condition is plausible by numerical test. Furthermore, the condition, written in terms of prime numbers, can be replaced by the condition written in terms of natural numbers. This implies that the problem of upper estimation of $l(p)$ is reduced to that on natural numbers.
\end{abstract}

\maketitle

\section{Introduction}

\begin{note}
	\begin{itemize}
		\item Let $\Z$ and $\P$ be the set of integers and prime numbers, respectively.
		\item Let $\N$ be $\{1,2,3,\cdots\}$, and its elements are called natural numbers.
		\item For real $x$, $\floor{x}$ means to the unique integer such that $\floor{x}\leq x<\floor{x}+1$; called the integer part of $x$.
		\item Unless stated otherwise, the greatest common divisor of two integers $m$ and $n$ is denoted by $(m,n)$.
	\end{itemize}
\end{note}

It is well known that arithmetical progressions are related to prime numbers. For example, the Dirichlet prime number theorem in $1837$ \cite{Dirichlet} and the Green-Tao theorem in $2004$ \cite{Green-Tao}. However, the following conjecture of Dickson has not been well studied.

\begin{cjt}[Dickson \cite{Dickson}]\label{cjt:Dickson}
	Let $k$ be a natural number. Suppose that $a_i$ and $b_i$ are integers with $b_i\geq1$ and $(a_i,b_i)=1$ for all $i=1,\ldots,k$. Then, there are infinitely many integers $n$ such that $a_1+b_1n,\cdots,a_k+b_kn$ are simultaneously primes.
\end{cjt}

If $k=1$, then this is nothing but the Dirichlet prime number theorem. However, all the other cases are unsolved. In $1962$, Bateman and Horn generalized this conjecture to irreducible polynomials.

\begin{cjt}[Bateman-Horn \cite{Bateman-Horn}]\label{cjt:Bateman-Horn}
	Suppose that $f_i(x)$ is an irreducible polynomial whose leading coefficient is a positive integer. Let $N$ be a natural number, and let $d_i$ be the degree of $f_i(x)$. Then, the number of $0<n\leq N$ such that $f_1(n),\cdots,f_k(n)$ are simultaneous by prime is given by
	\begin{align}\label{cjt:Bateman-Horn-1}
		\int_2^N\prod_{i=1}^k\frac{1}{\log f_i(x)}dx\cdot\prod_{p\in\P}\frac{1-w(p)/p}{\pas{1-p^{-1}}^k} \sim \frac{N}{d_1\cdots d_k\log^kN}\prod_{p\in\P}\frac{1-w(p)/p}{\pas{1-p^{-1}}^k}
	\end{align}
	where $w(p)$ is the number of $x\in\{0,1,\cdots, p-1\}$ which is a solution for $f_1(x)\times\cdots\times f_k(x)\equiv0\pmod p$.
\end{cjt}

Caldwell applied Conjecture.\ref{cjt:Bateman-Horn} to various progressions in \cite{Caldwell}. One of them is the Cunningham chain.

\begin{dfn}
	Let $p$ be a prime number. A sequence
	\[
	\Pas{p,2p+1,\cdots,\pas{p+1}2^{l_1(p)-1}-1}\;\pas{\textup{resp.}\Pas{p,2p-1,\cdots,\pas{p-1}2^{l_{-1}(p)-1}+1}}
	\]
	is called the first (\textup{resp.} second) Cunningham chain generated by $p$ if all elements are prime numbers and $(p+1)2^{l_1(p)-1}-1$ (\textup{resp.} $(p-1)2^{l_{-1}(p)-1}+1$) is composite. As we will see later, $l_1(p)$ and $l_{-1}(p)$ are finite. Also when we discuss the first case and the second case simultaneously, we write $l_1(p)$ and $l_{-1}(p)$ as $l(p)$.
\end{dfn}

If $p=2$, then
\begin{align*}
	2\times2+1=5\in\P,\;5\times2+1=11\in\P,\;11\times2+1&=23\in\P,\;23\times2+1=47\in\P,\;47\times2+1=95\not\in\P,\\
	2\times2-1=3\in\P,\;3\times2-1&=5\in\P,\;5\times2-1=9\not\in\P.
\end{align*}
Thus $l_1(2)=5, l_{-1}(2)=3$. If $p\not=2$, we have $l(p)\leq\textup{ord}\pas{p;2}$ by Fermat`s little theorem. Here, $\textup{ord}\pas{p;2}$ is the order of $2$ modulo $p$. In fact, in the case of the first Cunningham chain, we have
\[
\pas{p+1}2^{\textup{ord}\pas{p;2}}-1\equiv0\pmod p,
\]
and hence $(p+1)2^{\textup{ord}\pas{p;2}}-1$ is a composite number. The second case is similar. If we assume Artin`s conjecture (\text{cf.} \cite{Artin} pp.viii-x), there exist infinitely many primitive roots of $2$. Thus we only have $l(p)\leq p-1$ in this way. However, if we use the theorem of Euler, we obtain the following statement.

\begin{rem}
	$l(p)<p/2$ for all prime $p\geq7$.
\end{rem}
\begin{proof}
	First, we consider the case first Cunningham chain. If $p\equiv3\pmod4$, then $(p-1)/2$ is odd. Let $k=\varphi((p-1)/2)\geq2$. Here $\varphi$ denotes the Euler function. We have
	\[
	\pas{p+1}2^{k-1}\equiv2^k\equiv1 \pmod{\frac{p-1}{2}}.
	\]
	Thus $(p+1)2^{k-1}-1$ is a composite number for $p>3$, and we get $l_1(p)<k=\varphi((p-1)/2)<p/2$. If $p\equiv1\pmod4$,
	\[
	\pas{p+1}2^{k-1}\equiv2^{k+1}\equiv1\pmod{\frac{p-3}{2}}
	\]
	where $k=\varphi((p-3)/2)-1\geq2$, and hence $l_1(p)<\varphi((p-3)/2)-1<p/2$. Therefore $l_1(p)<p/2$ for all $p\geq7$. Let us next show the second case. If $p\equiv3\pmod4$, setting $k=\varphi((p+3)/2)-1\geq2$, we have
	\[
	\pas{p-1}2^{k-1}\equiv-2^{k+1}\equiv -1\pmod{\frac{p+3}{2}}.
	\]
	Thus,
	\[
	l_{-1}(p)<\varphi\pas{\frac{p+3}{2}}-1\leq \frac{p+3}{2}-2<\frac{p}{2}.
	\]
	Similarly, $(p-1)2^{\varphi((p+1)/2)-1}+1$ is composite if $p\equiv1\pmod4$.
\end{proof}

If $2p+1$ is also prime, then this prime number is called a safe prime. Its name comes from the fact that it is a prime number useful for Pollard`s $p-1$ algorithm (\textup{cf.}\cite{Pollard}) that suggests the vulnerability of the $RSA$ cryptosystem. It is useful if $p-1$ is decomposed into the product of small prime factors. Since $(2p+1)-1$ has $2$ and $p$ as prime factors and $p$ is big, for two safe primes $2p+1, 2q+1$, it is difficult to factorize $(2p+1)(2q+1)$ by the $p-1$ algorithm. However, the Lenstra elliptic-curve factorization in \cite{Lenstra} can overcome this weakness in the probabilistic sense. Thus the meaning of safe primes as the $RSA$ cryptosystem has faded. In $2004$, whose first draft is in $2002$, Agrawal, Kayal and Saxena found the $AKS$ primality test\cite{Agrawal-Kayal-Saxena}. In the present paper, they show that if in the case $k=2, f_1(n)=n, f_2(n)=2n+1$ of $(\ref{cjt:Bateman-Horn-1})$, the Sophie Germain prime density conjecture is true, the order of the complexity can be reduced to $\tilde{O}(\log^6 n)$. Here, $F(n)=\tilde{O}(G(n))$ is defined by $F(n)=O(G(n)h(\log G(n)))$ for some polynomial $h(x)$. But, in $2005$, Lenstra and Pomerance found a variant of $AKS$ that runs in $\tilde{O}(\log^6n)$ without any conjectures \cite{Lenstra-Pomerance}. Thus, now it is not necessary to use safe primes on the $AKS$ primality test. Safe primes can be effectively used in the Diffie-Hellman key exchange \cite{Diffie-Hellman}. Since this is the key to the strength of the discrete logarithm problem, we need to consider the question of how long and how well distributed Cunningham chains are.

We now let $k$ be a natural number. We consider the number $0<n\leq N$ such that $n,2n+1,\cdots,(n+1)2^{k-1}-1$ are simultaneously primes. Since
\[
M(n):=n(2n+1)\cdots\pas{(n+1)2^{k-1}-1}\equiv n\pmod 2,
\]
$n=0$ is the only solution of $M(n)\equiv0\pmod2$ for $n\in\{0,1\}$. Thus $w(2)=1$. If $p>2$, then $(n+1)2^{k-1}-1\pmod p$ has period $\textup{ord}\pas{p;2}$ with respect to $k$. Letting $v(p)=\min\{k,\textup{ord}\pas{p;2}\}$, this implies that $w(p)$ is equal to the number $w^\prime(p)$ of solutions of
\[
n\pas{2n+1}\cdots \pas{(n+1)2^{v(p)-1}-1}\equiv0\pmod p.
\]
We observe that each of $n,2n+1,\cdots,(n+1)2^{v(p)-1}-1$ has one solution in $\{0,\cdots,p-1\}$. That is, $w^\prime(p)\geq v(p)$. Assume that there exist $1\leq a\leq b\leq v(p)$ and $n\in\{0,\cdots,p-1\}$ satisfy
\[
\pas{n+1}2^{a-1}\equiv1\pmod p \quad\textup{ and } \quad \pas{n+1}2^{b-1}\equiv1\pmod p.
\]
Then, since $n+1$ and $p$ are relatively prime, $2^{a-1}\equiv2^{b-1}\pmod p$ and it implies $a=b$. We have $w(p)=w^\prime(p)=v(p)$. The same argument can be applied to the second Cunningham chains. From this argument and Conjecture.\ref{cjt:Bateman-Horn}, we can expect the following.

\begin{cjt}[Caldwell \cite{Caldwell}]\label{cjt:Caldwell}
	Fix a natural number $k$. Then
	\[
	\sum_{\substack{p\leq N\\ l(p)\geq k}}1\sim B_k\int_2^N\frac{dx}{\pas{\log x}\pas{\log 2x}\cdots\pas{\log2^{k-1}x}}\sim B_k\frac{N}{\log^kN}
	\]
	where the partial sum on the left hand side runs through all prime numbers $p\leq N$ satisfying $l(p)\geq k$. Here,
	\[
	B_k:=\prod_{p\in\P}\frac{1-w(p)/p}{\pas{1-p^{-1}}^k}=2^{k-1}\prod_{p>2}\frac{1-\min\Pas{k,\textup{ord}\pas{p;2}}/p}{\pas{1-p^{-1}}^k}.
	\]
\end{cjt}

If we examine the size of $B_k$, we can get some upper bound of $l(p)$.

\begin{prop}\label{prop:order of Bk}
	Let $\gamma$ be the Euler constant $\simeq0.57722$. Then for $k\geq2$,
	\[
	\log\log k+\gamma-2+O\pas{\frac{1}{\log k}}\leq \frac{\log B_k}{k}\leq \log k+\gamma+\log\log2-1+O\pas{\frac{\log k}{k}}.
	\]
\end{prop}
\begin{proof}
	The following facts in \cite{Rosser-Schoenfeld} will be used in the proof without mentioning explicitly:
	\begin{itemize}
		\item[$(a)$] $\displaystyle \log\log x+B-\frac{1}{2\log^2x}<\sum_{p\leq x}\frac{1}{p}<\log\log x +B+\frac{1}{\log^2x},$
		\item[$(b)$] $\displaystyle \prod_{p\leq x}\pas{1-\frac{1}{p}}=O\pas{\frac{1}{\log x}},$
		\item[$(c)$] $\displaystyle \prod_{p\leq x}\pas{1-\frac{1}{p}}^{-1}=e^\gamma\log x+O\pas{\frac{1}{\log x}},$
		\item[$(d)$] $\displaystyle \sum_{p\leq x}\log p<\pas{1+\frac{1}{2\log x}}x,$
	\end{itemize}
	where $x>2$ and $B=\gamma-\sum_{p\in\P}\sum_{k=2}^\infty k^{-1}p^{-k}$. For $k\geq2$ and a sufficiently large real number $x$, we define
	\[
	B_k(x) := 2^{k-1}\prod_{2<p\leq x}\frac{1-\min\Pas{k, \textup{ord}\pas{p;2}}/p}{\pas{1-p^{-1}}^k}.
	\]
	First, we estimate the lower bound of $B_k(x)$ in $k$. Write
	\begin{align*}
		B_k(x)=2^{k-1}\pas{\prod_{2<p\leq x}\pas{1-\frac{1}{p}}^{-k}}\pas{\prod_{2<p\leq k}\times\prod_{k<p\leq x}}\pas{1-\frac{\min\Pas{k, \textup{ord}\pas{p;2}}}{p}} =:2^{k-1}\Pi_1\times \Pi_2\times\Pi_3,
	\end{align*}
	say. Then $\log\Pi_1$, $\log\Pi_2$ and $\log\Pi_3$ are evaluated as
	\begin{align*}
		\log\Pi_1&=k\log\pas{\frac{1}{2}\prod_{p\leq x}\pas{1-\frac{1}{p}}^{-1}}=k\log\pas{\frac{e^\gamma}{2}\log x+O\pas{\frac{1}{\log x}}}\\
		&=k\log\frac{e^\gamma}{2}+k\log\log x+O\pas{\frac{k}{\log^2x}},\\
		\log\Pi_2&\geq\log\prod_{2<p\leq k}\pas{1-\frac{p-1}{p}}=\log2-\sum_{p\leq k}\log p=-k-\frac{k}{\log k}+\log2,\\
		\log\Pi_3&\geq\prod_{k<p\leq x}\pas{1-\frac{k}{p}}=-\sum_{k<p\leq x}\sum_{n=1}^\infty\frac{1}{n}\pas{\frac{k}{p}^n},
	\end{align*}
	respectively. Since
	\begin{align*}
		k\sum_{k<p\leq x}\frac{1}{p}&\leq\pas{\log\pas{\frac{\log x}{\log k}}+\frac{1}{\log^2x}+\frac{1}{2\log^2k}}k
	\end{align*}
	and
	\begin{align*}
		\sum_{k<p\leq x}\sum_{n=2}^\infty\frac{1}{n}\pas{\frac{k}{p}}^n
		&=\sum_{n=2}^\infty\frac{k^n}{n}\sum_{k<p\leq x}\frac{1}{p^n}<\sum_{n=2}^\infty\frac{k^n}{n}\sum_{m=k+1}^\infty\frac{1}{m^n}\\
		&<\sum_{n=2}^\infty\frac{k^n}{n}\pas{\frac{1}{\pas{k+1}^n}+\frac{1}{n-1}\frac{1}{\pas{k+1}^{n-1}}}\\
		&<k\sum_{n=2}^\infty\frac{1}{n}\pas{\frac{1}{n+k}+\frac{1}{n-1}}\\
		&=k+O\pas{\log k},
	\end{align*}
	we have
	\begin{align*}
		\log\Pi_3\geq-k\pas{\log\pas{\frac{\log x}{\log k}}+\frac{1}{\log^2x}+\frac{1}{2\log^2k}}-k+O\pas{\log k}.
	\end{align*}
	Therefore
	\begin{align*}
		\frac{1}{k}\log B_k(x)
		&=\frac{1}{k}\pas{\pas{k-1}\log2+\log\Pi_1+\log\Pi_2+\log\Pi_3}\\
		&>\frac{k-1}{k}\log2++\log\frac{e^\gamma}{2}+\log\log x+O\pas{\frac{1}{\log^2x}}+\frac{\log2}{k}-1-\frac{1}{2\log k}\\
		&\quad-\log\pas{\frac{\log x}{\log k}}-\frac{1}{\log^2x}-\frac{1}{2\log^2k}-1+O\pas{\frac{\log k}{k}}\\
		&=\log\log k+\gamma-2-\frac{1}{2\log k}-\frac{1}{2\log^2k}+O\pas{\frac{1}{\log^2x}}+O\pas{\frac{\log k}{k}}
	\end{align*}
	and we have
	\begin{align*}
		\frac{1}{k}\log B_k\geq \log\log k+\gamma-2-\frac{1}{2\log k}-\frac{1}{2\log^2k}+O\pas{\frac{\log k}{k}}
	\end{align*}
	as $x\to\infty$. We next consider the upper bound of $B_k(x)$ in $k$. Write
	\begin{align*}
		B_k(x)
		&=2^{k-1}\pas{\prod_{2<p\leq x}\pas{1-\frac{1}{p}}^{-k}}\pas{\prod_{2<p\leq k}\times\prod_{k<p\leq 2^k}\times\prod_{2^k<p\leq x}}\pas{1-\frac{\min\Pas{k,\textup{ord}\pas{p;2}}}{p}},\\
		&=:2^{k-1}\Pi_1\times\Pi_2\times\Pi_3\times\Pi_4
	\end{align*}
	say. Then $\log\Pi_1, \log\Pi_2$ and $\log\Pi_4$ are evaluated as
	\begin{align*}
		\log\Pi_1&
		=k\log\frac{e^\gamma}{2}+k\log\log x+O\pas{\frac{k}{\log^2x}},\\
		\log\Pi_2&
		\leq\log\prod_{2<p\leq k}\pas{1-\frac{1}{p}}=O\pas{\log\log k},\\
		\log\Pi_4&
		\leq\log\prod_{2^k<p\leq x}\pas{1-\frac{k}{p}}=-\sum_{2^k<p\leq x}\sum_{n=1}^\infty\frac{1}{n}\pas{\frac{k}{p}}^n<-k\sum_{2^k<p\leq x}\frac{1}{p}\\
		&<-k\pas{\log\pas{\frac{\log x}{k\log2}}-\frac{1}{k^2\log2}-\frac{1}{2\log^2x}}\\
		&=k\log\pas{\frac{k\log2}{\log x}}+\frac{k}{2\log^2x}+O\pas{\frac{1}{k}},
	\end{align*}
	since $\textup{ord}\pas{p;2}\geq k$ where $p>2^k$. If $p>2^{j-1}$ for some $j\in\N$, then $\textup{ord}\pas{p;2}\geq j$. Thus, putting $y=\log k/\log2$, we obtain
	\begin{align*}
		\log\Pi_3
		&\leq\log\prod_{y<j\leq k}\prod_{2^{j-1}<p\leq 2^j}\pas{1-\frac{j}{p}}\\
		&<\log\prod_{y<j\leq k}\pas{\frac{j-1}{j}}^j\exp\pas{\frac{1}{2j\log^22}+\frac{j}{\pas{j-1}^2\log^22}}\\
		&=\log\pas{\frac{\floor{y}^{\floor{y}}}{\pas{\floor{y}-1}!}\cdot\frac{\pas{k-1}!}{k^k}}+\frac{1}{\log^22}\sum_{y<j\leq k}\pas{\frac{1}{2j}+\frac{j}{\pas{j-1}^2}}.
	\end{align*}
	Applying $\log(n^n/n!)=n+O\pas{\log n}$, we find that the right-hand side of the above is
	\begin{align*}
		=\log\frac{k!}{k^k}+O\pas{\log k}=-k+O\pas{\log k}.
	\end{align*}
	Therefore, we have
	\begin{align*}
		\frac{1}{k}\log B_k(x)
		&=\frac{1}{k}\pas{\pas{k-1}\log2+\log\Pi_1+\log\Pi_2+\log\Pi_3+\log\Pi_4}\\
		&\leq\frac{k-1}{k}\log2+\log\frac{e^\gamma}{2}+\log\log x+O\pas{\frac{1}{\log^2x}}+O\pas{\frac{\log\log k}{k}}\\
		&\quad-1+O\pas{\frac{\log k}{k}}+\log\pas{\frac{k\log2}{\log x}}+\frac{1}{2\log^2x}+O\pas{\frac{1}{k^2}}\\
		&=\log k+\gamma+\log\log2-1+O\pas{\frac{1}{\log^2x}}+O\pas{\frac{\log k}{k}}.
	\end{align*}
\end{proof}

Let $k(N)$ be the maximum of the length of the Cunningham chain which is generated by $p\leq N$, that is,
\[
k(N)=\max\Pas{l(p)\colon p\leq N}.
\]
Let $f(k)=\frac{1}{k}\log B_k$. Then we have
\[
B_{k(N)}\frac{N}{\pas{\log N}^{k(N)}}=\pas{N^{\frac{1}{k(N)}}\frac{f(k(N))}{\log N}}^{k(N)}=:F(N)^{k(N)}.
\]
Suppose that there exists a sequence $\{N_n\}_{n=1}^\infty$ satisfying
\[
\pas{1-F(N_n)}k(N_n)\gg1.
\]
Then, there exists a constant $C>0$ which satisfies $F(N_n)\leq1-C/k(N_n)$. Thus
\begin{align*}
	\limsup_{n\to\infty}B_{k(N_n)}\frac{N_n}{\pas{\log N_n}^{k(N_n)}}
	&\leq \lim_{n\to\infty}\pas{1-\frac{C}{k(N_n)}}^{k(N_n)}\\
	&=\begin{cases}
		e^{-C}		&\text{if}	\quad	\lim_{n\to\infty}k(N_n)=+\infty,\\
		\max_{n\in\N}\pas{1-\frac{C}{k(N_n)}}^{k(N_n)}		&\text{if}	\quad	\lim_{n\to\infty}k(N_n)<+\infty
	\end{cases}\\
	&<1.
\end{align*}
However, unless the order with respect to $k$ of the terms that vanish by approximation of Conjecture.\ref{cjt:Caldwell} is small, this result contradicts Conjecture.\ref{cjt:Caldwell} and the maximality of $k(N)$. Therefore we may expect that $\lim_{N\to\infty}(1-F(N))k(N)=0$. In particular, since $\lim_{N\to\infty}F(N)=1$, we have $F(N)<2$ for sufficiently large $N$. Thus
\[
\frac{\log N}{k(N)}+\log f(k(N))-\log\log N<\log2.
\]
From Proposition.\ref{prop:order of Bk} and $l(2)\geq3$, we have
\[
\log f(k(N))>\log\log k(N)-2>\log\log3-2.
\]
Therefore, we obtain
\[
k(N)\geq\pas{1+o(1)}\frac{\log N}{\log\log N}.
\]
From this, we may conjecture the following:

\begin{cjt}\label{cjt:omega order of Cunningham chains}
	\[
	l(p)=\Omega\pas{\frac{\log p}{\log\log p}}\quad \textup{ on }\P.
	\]
\end{cjt}

In \cite{Augustin} it is reported that, $p_1:=2759832934171386593519$ is the first term of the longest first Cunningham chain in the data up to $2020$. Its length is $17$. And $p_2=42008163485623434922152331$ is the first term of the longest second whose length is $19$. Thus
\begin{align*}
	l_1(p_1)=17, \;\frac{\log p_1}{\log\log p_1}\simeq12.661,\\
	l_{-1}(p_2)=19, \;\frac{\log p_2}{\log\log p_2}\simeq14.470.
\end{align*}
In this way, a better estimation of $B_k$ implies a better bound of the length of Cunningham chains. However, it is still unknown whether $\limsup_{p\to\infty}l(p)/p=0$ or not.

In this paper, by using a generalized Fibonacci sequence, we get $l(p)\ll\log p$ under a certain condition $(Theorem.\ref{thm:related with CC})$. It seems that this sufficient condition is plausible by numerical test. The condition can be extended from prime numbers to natural numbers $(Corollary.\ref{crl:related with CC})$. This implies that the problem of upper estimation of $l(p)$ is reduced to that on natural numbers. One of the benefit of this is that we can use methods of number theory to solve it.

\textbf{Acknowledgement.} I wish to express my gratitude to Professor Kohji Matsumoto and Dr. Sh\={o}ta Inoue for their advice towards the present research. I am also thankful to Mr. Yusei Ishida for assisting in writing computer programs. And I thank Dr. Kenta Endo, Dr. Kota Saito, Mr. Yuichiro Toma, Mr. Hideki Hunakura and Mr. Tomohiro Fukada for many useful discussions.

\section{$\Falpha$ numbers}
\begin{dfn}
	Let $\alpha\geq1$ be a natural number. A sequence $\Falpha$:
	\[
	F_0=0, \;F_1=1, \;F_{n+1}=\alpha F_n+F_{n-1}\quad \pas{n\in\Z}
	\]
	is called the generalized Fibonacci sequence, and its elements are called $\Falpha$ numbers.
	For example, enumerating $F_{-5}$ through $F_5$ we have
	\[
	\cdots,\alpha^4+3\alpha^2+1,-\alpha^3-2\alpha,\alpha^2+1,-\alpha,1,0,1,\alpha,\alpha^2+1,\alpha^3+2\alpha,\alpha^4+3\alpha^2+1,\cdots.
	\]
\end{dfn}

The following facts on $\Falpha$ numbers are well known. For example, in \cite{Koshy}, T. Koshy showed the case $\alpha=1$ of $(a)$ in Section 5.2 pp.88-90, $(b)$ in Section 5 p.82, $(c)$ in Section 5.6 p.103, $(d)$ in Section 20.1 p.397, $(e)$ in Section 10.1 pp.171-173, and $(f)$ in Section 10.1 pp.173-174. We can similarly show the case $\alpha>1$.

\begin{fact}\label{fact:fact of Fa numbers}
	Let $m,n$ be integers, and put $\valpha=(\alpha+\sqrt{\alpha^2+4})/2$. We have:
	\begin{itemize}
		\item[$(a)$] $\displaystyle F_n=\frac{1}{\sqrt{\alpha^2+4}}\pas{\valpha^n-\pas{-\valpha}^{-n}},$
		\item[$(b)$] $\displaystyle \sum_{i=1}^{n}F_i=\pas{1+\frac{1}{\alpha}}F_n+\frac{F_{n-1}-1}{\alpha}<\pas{1+\frac{2}{\alpha}}F_n\quad \pas{n\geq1},$
		\item[$(c)$] $\displaystyle F_{-n}=(-1)^{n+1}F_n,$
		\item[$(d)$] $\displaystyle F_{m+n}=F_{m+1}F_n+F_mF_{n-1},$
		\item[$(e)$] $\displaystyle m\mid n \iff F_m\mid F_n,$
		\item[$(f)$] $\displaystyle (F_m,F_n)=F_{(m,n)}.$
	\end{itemize}
\end{fact}

If $\alpha=1$, then $m$ should not be $2$. Indeed, $F_2=1, F_3=2$ thus $F_2\mid F_3$, but $2\nmid3$. We next define the divisor function on $\Falpha$.

\begin{dfn}
	In this article, a natural number $d$ is called a $\Falpha$ divisor of $n$ if $d\in\Falpha$ and $d\mid n$. A map $\fdf:\N\to \C$ is defined by
	\[
	\FDF{n}:= \sum_{\substack{d\mid n \\ 0<d\in\Falpha}}d
	\]
	for $n\in\N$, and called the $\Falpha$ divisor function. Let $\fdf^1=\fdf$, and $\fdf^k=\fdf\circ\fdf^{k-1}$ for $k>1$.
\end{dfn}

In Section.$3$, we will investigate the relationship between the iteration of $\fdf$ and Cunningham chains.

\begin{ex}
	Suppose $\alpha=3$. Since
	\[
	\Falpha: \cdots, 0, 1, 3, 10, 33, 109, 360, \cdots,
	\]
	we get
	\begin{align*}
		\FDF{2}&=1, \;\FDF{3}=4, \;\FDF{4}=1, \\
		\FDFk{3}{109}&=\FDFk{2}{110}=\FDF{11}=1.
	\end{align*}
\end{ex}

We consider the Dirichlet series associated with $\fdf$. Put $\zeta_\alpha(s)=\sum_{0<n\in\Falpha}n^{-s}$ for $\alpha\geq1$. Then $\zeta_\alpha(s)$ converges for all $s$ with $\Re(s)>0$. Suppose $f(n)$ is $n$ if $n\in\Falpha$, and is $0$ otherwise. Then
\begin{align}\label{eq:Dirichlet product}
	\zeta(s)\zeta_\alpha(s-1)=\pas{\sum_{m=1}^\infty\frac{1}{m^s}}\pas{\sum_{0<n\in\Falpha}\frac{1}{n^{s-1}}}=\sum_{m,n=1}^\infty\frac{f(n)}{(mn)^s}=\sum_{n=1}^\infty\frac{1}{n^s}\sum_{\substack{d\mid n \\ 0<d}}f(d)=\sum_{n=1}^\infty\frac{\FDF{n}}{n^s}
\end{align}
for $\Re(s)>1$. As we can see from this expression, the research of $\fdf$ will be useful for the study of $\zeta_\alpha$. In particular, $\zeta_1$ is called the Fibonacci zeta function by Egami\cite{Egami} and Navas\cite{Navas}, and it is a meromorphic function on $\C$. It is the famous unsolved problem whether $\zeta_1(1)$ is transcendental or not.

Hereafter, we suppose that $\alpha\geq3$ unless explicitly stated otherwise. Let $\ind{n}$ be the index of the maximal $\Falpha$ number $\leq n$ for $n\in\N$, that is, if we take $k$ satisfying $F_k\leq n<F_{k+1}$, then $\ind{n}=k$. Since the sequences $\Falpha$ are increasing for arbitrary $\alpha\geq3$, $\ind{n}$ is determined uniquely for all $n$. In addition, $\ind{n}$ is not decreasing for $n$.

\begin{lem}\label{lem:there exists a k s.t. FDFk(n)=1}
	Let $k^\prime:=\ind{\FDF{n}}$. Then $k^\prime\leq\ind{n}$ for all $n$, and if $\FDF{n}\neq1$, then
	\[
	F_{k^\prime}<\FDF{n}<\pas{1+\frac{2}{\alpha}}F_{k^\prime}.
	\]
	In particular, $F_{k^\prime}\nmid\FDF{n}$.
\end{lem}
\begin{proof}
	$k^\prime=1\leq\ind{n}$ if $\FDF{n}=1$. Suppose that $\FDF{n}>1$. Let $i_0$ be the maximal $i$ with $F_i\mid n$. Then $i_0>1$ and $n$ has at least two $\Falpha$ divisors $F_{i_0}$ and $F_1$. Thus we estimate that
	\begin{align*}
		F_{i_0}<F_{i_0}+F_1\leq\FDF{n}\leq \sum_{i=1}^{i_0}F_i<\pas{1+\frac{2}{\alpha}}F_{i_0}<2F_{i_0}<F_{i_0+1}
	\end{align*}
	from Fact.\ref{fact:fact of Fa numbers} $(b)$. This implies that $k^\prime=i_0=\ind{F_{i_0}}\leq\ind{n}$.
\end{proof}

Let $\FDF{\N}$ be the set of all $\FDF{n}$ for $n\in\N$.

\begin{thm}\label{thm:there exists a k s.t. FDFk(n)=1}
	For every natural number $n$, there exists a non-negative integer $k\leq\log n /\log\valpha +2$ such that $\FDFk{k}{n}=1$.
\end{thm}

In Section.$4$, we will improve this bound to $k\ll \log\log n$.

\begin{proof}
	Let $a$ be an element of $\FDF{\N}$, and let $k=\ind{a}$. If $k=1$, then $\FDFk{0}{a}=a=1$. In addition, since
	\[
	F_2+F_1=\alpha+1\not\equiv0\pmod\alpha,
	\]
	we get $\FDF{a}=1$ if $k=2$. Suppose that $k\geq3$. Since $F_k\nmid a$ from Lemma.\ref{lem:there exists a k s.t. FDFk(n)=1}, we have
	\[
	\FDF{a}\leq \sum_{i=1}^{k-1}F_i<\pas{1+\frac{2}{\alpha}}F_{k-1}<F_k.
	\]
	Thus $\ind{\FDF{a}}\leq k-1$. By iterating this argument, we have $\ind{\FDFk{k-2}{a}}\leq2$. Therefore $\FDFk{k-1}{a}=1$. This also holds for $k=1,2$. Let $k^\prime=\ind{\FDF{n}}$ for arbitrary $n$. We have
	\[
	\FDFk{k^\prime}{n}=\FDFk{k^\prime-1}{\FDF{n}}=1.
	\]
	Since $k^\prime\leq\ind{n}$ from Lemma.\ref{lem:there exists a k s.t. FDFk(n)=1}, this leads to $\FDFk{\ind{n}}{n}=1$. On the other hand, since
	\[
	\valpha^{\ind{n}}=F_{\ind{n}}\sqrt{\alpha^2+4}+\pas{-\valpha}^{-\ind{n}}
	\]
	from Fact.\ref{fact:fact of Fa numbers} $(a)$ and $F_{\ind{n}}\leq n$, we obtain that
	\[
	\ind{n}\leq \frac{1}{\log\valpha}\pas{\log n+\log\pas{\sqrt{\alpha^2+4}+\frac{\pas{-\valpha}^{-\ind{n}}}{n}}}.
	\]
	Now, $n$ is a natural number and $\valpha=(\alpha+\sqrt{\alpha^2+4})/2$ is bigger than $\alpha\geq3$. Thus
	\begin{align*}
		\frac{1}{\log\valpha}\log\pas{\sqrt{\alpha^2+4}+\frac{\pas{-\valpha}^{-\ind{n}}}{n}}
		&<\frac{1}{\log\alpha}\log\pas{\sqrt{\alpha^2+4}+\frac{1}{\alpha}}\\
		&=\frac{1}{\log\alpha}\log\pas{\sqrt{1+\frac{4}{\alpha^2}}+\frac{1}{\alpha^2}}+1\\
		&<\flog{\alpha}\log\pas{1+\frac{5}{\alpha^2}}+1<2.
	\end{align*}
	This implies $\ind{n}<\log n/\log\valpha+2$.
\end{proof}

Here, we define
\[
\ord{n}=\min\Pas{k\in\Z_{\geq0}\colon \fdf^k(n)=1}.
\]
In Section.$4$, we will find that a better estimation of $\ord{n}$ for $n$ implies a better bound of the length of Cunningham chains.

\section{The relation between $\Falpha$ numbers and the Cunningham chains}
In this section, let $m, n$ be natural numbers. From Fact.\ref{fact:fact of Fa numbers} $(c)$ and $(d)$,
\begin{align*}
	F_{m+n}+F_{m-n}
	&=F_{m+1}F_n+F_mF_{n-1}+F_mF_{-n+1}+F_{m-1}F_{-n}\\
	&=\pas{1+(-1)^n}F_mF_{n-1}+\pas{F_{m+1}+(-1)^{n-1}F_{m-1}}F_n\\
	&=\begin{cases}
		F_n\pas{F_{m+1}+F_{m-1}} & \pas{n \textup{ is odd}},\\
		2F_mF_{n-1}+\pas{F_{m+1}-F_{m-1}}F_n & \pas{n \textup{ is even}}.
	\end{cases}
\end{align*}
Since $2F_{m+1}F_n=F_{m+1}F_n+\alpha F_mF_n+F_{m-1}F_n$, we have
\begin{align}\label{eq:summation formula-1}
	F_{m+n}+F_{m-n}=
	\begin{cases}
		F_n\pas{F_{m+1}+F_{m-1}} & \pas{n \textup{ is odd}},\\
		F_m\pas{F_{n-1}+F_{n+1}} & \pas{n \textup{ is even}}.
	\end{cases}
\end{align}
Note that $F_{m+1}+F_{m-1}=F_{2m}/F_m$ from Fact.\ref{fact:fact of Fa numbers} $(d)$.

\begin{lem}\label{lem:divisors of F(m+1)+F(m-1)}
	The $\Falpha$ divisors of $F_{m+1}+F_{m-1}$ are at most $1$ and $\alpha$. Moreover, $\alpha\mid F_{m+1}+F_{m-1}$ if and only if $m\equiv1\pmod2$, that is,
	\begin{align*}
		\FDF{F_{m+1}+F_{m-1}}=
		\begin{cases}
			1               & (m\textup{ is even}), \\
			\alpha+1       & (m\textup{ is odd}).
		\end{cases}
	\end{align*}
\end{lem}
\begin{proof}
	If $\alpha\mid2F_m$, then
	\[
	3\leq\alpha=\pas{\alpha,2F_m}\leq2\pas{\alpha,F_m}=2F_{\pas{2,m}},
	\]
	and hence $2\mid n$. Conversely, $\alpha\mid F_m$ if $m$ is even from Fact.\ref{fact:fact of Fa numbers} $(e)$. Therefore
	\[
	\alpha\mid 2F_m \iff m \equiv0 \pmod2.
	\]
	Also since $F_{m+1}+F_{m-1}=\alpha F_m+2F_{m-1}$, we find that
	\[
	\alpha\mid F_{m+1}+F_{m-1}\iff \alpha\mid2F_{m-1}\iff m\equiv1\pmod 2.
	\]
	Here, take a natural number $d$ that satisfies $F_d\mid F_{m+1}+F_{m-1}(=F_{2m}/F_m)$. Since $F_d\mid F_{2m}$, there exists a natural number $q$ such that $2m=qd$. Suppose that $d$ is odd. Then $q$ is even and $F_d\mid2F_{\frac{q}{2}d-1}$ since $F_d\mid F_{\frac{q}{2}d}$ from Fact.\ref{fact:fact of Fa numbers} $(e)$. In addition,
	\[
	\pas{F_d,F_{\frac{q}{2}d-1}}=F_{\pas{d,\frac{q}{2}d-1}}=F_1=1.
	\]
	Thus, we get $F_d\mid2$. This and $\alpha\geq3$ imply $d=1$. Let us next suppose that $d$ is even. Then $F_{\frac{d}{2}}\mid F_d$ from Fact.\ref{fact:fact of Fa numbers} $(e)$, and this implies $F_{\frac{d}{2}}\mid2F_{q\frac{d}{2}-1}$. Similarly, since $(F_{\frac{d}{2}}, F_{q\frac{d}{2}-1})=1$, we have $F_{\frac{d}{2}}\mid2$, and hence $d=2$. From these results, we conclude that there is no $d\geq3$ which satisfies $F_d\mid F_{m+1}+F_{m-1}$.
\end{proof}

\begin{thm}\label{thm:divisors of F(m+1)+F(m-1)}
	Let $p$ be a prime number. If $p\not=m<2p$, then
	\begin{align*}
		\FDF{F_{m+p}+F_{m-p}}=
		\begin{cases}
			\begin{cases}
				1  & \pas{m=1}\\
				F_3+1  & \pas{m=3}
			\end{cases}  & \pas{p=2},\\
			F_p+
			\begin{cases}
				1       & (m\textup{ is even})\\
				\alpha+1 & (m\textup{ is odd})
			\end{cases}.  & \pas{p\neq2}.
		\end{cases}
	\end{align*}
\end{thm}
\begin{proof}
	Note that $F_{m+p}+F_{m-p}=F_p\cdot F_{2m}/F_m$ from $(\ref{eq:summation formula-1})$. If $p=2$,
	\begin{align*}
		F_{m+p}+F_{m-p}=
		\begin{cases}
			\alpha^2+2 	&		(m=1),\\
			\pas{\alpha^2+2}F_3.	&	(m=3).
		\end{cases}
	\end{align*}
	Suppose that $p>2$. We have
	\[
	\pas{F_p,\frac{F_{2m}}{F_m}}\leq\pas{F_p,F_{2m}}=F_{\pas{p,2m}}=F_1=1
	\]
	that is $(F_p, F_{2m}/F_m)=1$. Here, let $D_\alpha(n)$ be the set of all $\Falpha$ divisors of $n$. We find that
	\[
	D_\alpha(F_p)\cap D_\alpha\pas{\frac{F_{2m}}{F_m}}=\Pas{1}.
	\]
	And in general,
	\[
	D_\alpha\pas{F_p\frac{F_{2m}}{F_m}}\supset D_\alpha(F_p)\cup D_\alpha\pas{\frac{F_{2m}}{F_m}}.
	\]
	We will show the inclusion relation of the reverse direction. Take an arbitrary $F_s$ in $D_\alpha(F_p\cdot F_{2m}/F_m)$. Then there exists $a$ and $b$ satisfying $a\mid F_p, b\mid F_{2m}/F_m$ and $ab=F_s$. Since $(b,F_p)=1$,
	\[
	a=\pas{a,F_p}=\pas{a,F_p}\pas{b,F_p}=\pas{ab,aF_p,bF_p,F_p^2}=\pas{ab,\pas{a,b,F_p}F_p}=\pas{ab,F_p}=F_{\pas{s,p}}=1\textup{ or }F_p.
	\]
	In the case $a=1$, $F_s=b\in D_\alpha(F_{2m}/F_m)$ holds. Thus we have $ab\in\FDF{F_{2m}/F_m}$. Suppose that $a=F_p$. Then $k:=s/p$ is a natural number from Fact.\ref{fact:fact of Fa numbers} $(e)$, and $F_{kp}/F_p\mid F_{2m}/F_m$ holds. If $m<p$, then we have $k=1$ since $F_{kp}/F_p>F_{2m}/F_m$ for $k\geq2$. We consider the case $p<m<2p$. Note that
	\[
	\frac{F_{2n}}{F_n}=F_{n+1}+F_{n-1}\leq F_{n+2}+F_n=\frac{F_{2(n+1)}}{F_{n+1}}
	\]
	for all $n\in\N$. Thus we obtain that
	\begin{align*}
		\frac{F_{2m}}{F_m}\leq \frac{F_{2(m+1)}}{F_{m+1}}\leq \cdots\leq \frac{F_{2(2p-1)}}{F_{2p-1}}
		&=F_{2p}+F_{2p-2}\\
		&<F_{2p}+F_{2p-2}+F_{p-1}\frac{F_{2p}}{F_p}+F_{2p-3}\\
		&=\frac{\pas{F_p+F_{p-1}}F_{2p}+F_p\pas{F_{2p-2}+F_{2p-3}}}{F_p}\\
		&<\frac{F_{p+1}F_{2p}+F_pF_{2p-1}}{F_p}\\
		&=\frac{F_{3p}}{F_p}.
	\end{align*}
	This implies that $k\leq2$. We suppose $k=2$ and let $r=m-p$. Then $1\leq r<p$ and
	\begin{align*}
		\frac{F_{2m}}{F_m}
		&=F_{m+1}+F_{m-1}=F_{p+r+1}+F_{p+r-1}=F_{p+1}F_{r+1}+2F_pF_r+F_{p-1}F_{r-1}\\
		&=\pas{F_{p+1}+F_{p-1}}F_{r+1}+2F_pF_r+F_{p-1}F_{r-1}-F_{p-1}F_{r+1}\\
		&=\frac{F_{2p}}{F_p}F_{r+1}+F_pF_r+\pas{\alpha F_{p-1}+F_{p-2}}F_r+F_{p-1}F_{r-1}-\pas{\alpha F_{p-1}F_r+F_{p-1}F_{r-1}}\\
		&=\frac{F_{2p}}{F_p}F_{r+1}+\frac{F_{2(p-1)}}{F_{p-1}}F_r.
	\end{align*}
	Since $F_{2p}/F_p\mid F_{2m}/F_m$, we have $F_{2p}/F_p\mid F_{2(p-1)}/F_{p-1}\cdot F_r$. Thus
	\begin{align*}
		\frac{F_{2p}}{F_p}
		&=\pas{\frac{F_{2p}}{F_p},\frac{F_{2(p-1)}}{F_{p-1}}F_r}\leq\pas{\frac{F_{2p}}{F_p},\frac{F_{2(p-1)}}{F_{p-1}}}\pas{\frac{F_{2p}}{F_p},F_r}\\
		&\leq\pas{F_{2p},F_{2(p-1)}}\pas{F_{2p},F_r}=F_{2(p,p-1)}F_{\pas{2p,r}}\leq F_2^2=\alpha^2.
	\end{align*}
	This implies $F_{2p}\leq \alpha^2 F_p$ whose right-hand side does not exceed $F_{p+2}$, and hence $p\leq2$. However, this contradicts $p>3$, and we get $k=1$. In other words, $b=1$ in the case $a=F_p$, and then $ab\in D_\alpha(F_p)$. From these result, we have
	\begin{align*}
		D_\alpha(F_p)\cap D_\alpha\pas{\frac{F_{2m}}{F_m}}=\Pas{1} \text{ and } D_\alpha\pas{F_p\frac{F_{2m}}{F_m}}=D_\alpha(F_p)\cup D_\alpha\pas{\frac{F_{2m}}{F_m}},
	\end{align*}
	and hence
	\begin{align*}
		\FDF{F_{m+p}+F_{m-p}}
		=\sum_{d\in D_\alpha\pas{F_p\frac{F_{2m}}{F_m}}}d
		=\pas{\sum_{d\in D_\alpha\pas{F_p}}+\sum_{d\in D_\alpha\pas{\frac{F_{2m}}{F_m}}}}d-1
		=\FDF{F_p}+\FDF{\frac{F_{2m}}{F_m}}-1.
	\end{align*}
	Apply Lemma.\ref{lem:divisors of F(m+1)+F(m-1)} to complete the proof.
\end{proof}

From this theorem, we find the following corollary. That is a relationship between the iteration of $\fdf$ and Cunningham chains.

\begin{crl}\label{crl:divisors of F(2p+1)+1}
	For every odd prime $p$,
	\[
	\FDF{F_{2p\pm1}+1}=F_p+1.
	\]
	In particular, if $2p\pm1$ is also prime, then
	\begin{align}\label{crl:divisors of F(2p+1)+1-1}
		\FDFk{2}{F_{2p\pm1}}=\FDF{F_p}.
	\end{align}
	Further, by iterating this argument, we obtain
	\[
	l_{\pm1}(p)-1=\ord{F_{(p\pm1)2^{l_{\pm1}(p)-1}\mp1}}-\ord{F_p}.
	\]
\end{crl}
\begin{proof}
	It follows $m=p\pm1$ in Theorem.\ref{thm:divisors of F(m+1)+F(m-1)}. Further
	\[
	\ord{F_{(p\pm1)2^{l_{\pm1}(p)-1}\mp1}}=1+\ord{F_{(p\pm1)2^{l_{\pm1}(p)-2}\mp1}}=\cdots=l_{\pm1}(p)-1+\ord{F_p}.
	\]
\end{proof}

In Section.$4$, we will show the converse of $(\ref{crl:divisors of F(2p+1)+1-1})$.

\section{The upper bound of $\ord{n}$}
In this section, our aim is to prove Theorem.\ref{thm:converse of cor.3.4} and Theorem.\ref{thm:related with CC}. Those are important results that suggest the relationship between $\fdf$ and Cunningham chains. In the process of the proof, we will use the following theorem which is called the generalized Zeckendorf`s theorem by Hoggatt \cite{Hoggatt} and Keller \cite{Keller}.

\begin{dfn}[Zeckendorf-Hoggatt-Keller]
	Every natural number $n$ has the unique representation:
	\[
	n=\sum_{i=1}^ra_iF_{c_i}
	\]
	where $r$ is a natural number and sequences $\{a_i\}_{i=1}^r, \{c_i\}_{i=1}^r\subset\N$ satisfy the following conditions.
	\begin{enumerate}
		\item[(i)] $\displaystyle 0<a_i\leq\alpha\;\pas{i=1,\ldots,r}, \;a_1<\alpha$,
		\item[(ii)] $\displaystyle 1\leq c_1<\cdots<c_r$,
		\item[(iii)] $\displaystyle c_{i+1}=c_i+1\Longrightarrow a_{i+1}<\alpha \;\pas{i=1,\ldots,r-1}$.
	\end{enumerate}
	Such a sum is called the Zeckendorf representation of $n$.
\end{dfn}

\begin{rem}
	By definition, $\sum_{\substack{d\mid n \\ 0<d\in\Falpha}}d$ is the Zeckendorf representation of $\FDF{n}$.
\end{rem}

\begin{lem}\label{lem:fractional part of Falpha}
	For every $k>i\geq0$,
	\begin{align}
		F_{k+i}&\equiv \pas{-1}^{i+1}F_{k-i} \pmod{F_k}, \label{lem:fractional part of Falpha-1}\\
		F_{2k+i}&\equiv \pas{-1}^kF_i \pmod{F_k}, \label{lem:fractional part of Falpha-2}\\
		F_{3k+i}&\equiv \pas{-1}^{k+i+1}F_{k-i} \pmod{F_k}, \label{lem:fractional part of Falpha-3}\\
		F_{4k+i}&\equiv F_i \pmod{F_k}. \label{lem:fractional part of Falpha-4}
	\end{align}
	More generally, for every two non-negative integers $a,b$ with $a\equiv b\pmod4$,
	\begin{align}\label{lem:fractional part of Falpha-5}
		F_{ak+i}\equiv F_{bk+i} \pmod{F_k}.
	\end{align}
\end{lem}
\begin{proof}
	Fix a natural number $k$. First, we prove $(\ref{lem:fractional part of Falpha-1})$. From Fact.\ref{fact:fact of Fa numbers} $(d)$,
	\[
	F_{k+i}=F_kF_{i+1}+F_{k-1}F_i\equiv F_{k-1}F_i \pmod{F_k}.
	\]
	Thus $(\ref{lem:fractional part of Falpha-1})$ holds for $i=0,1$. Suppose that $(\ref{lem:fractional part of Falpha-1})$ holds for all natural numbers less than $i\geq2$. The right hand side is
	\begin{align*}
		\pas{-1}^{i+1}F_{k-i}=\pas{-1}^{i+1}\pas{F_{k-i+2}-\alpha F_{k-i+1}}=\pas{-1}^{i-1}F_{k-i+2}+\alpha \pas{-1}^iF_{k-i+1}.
	\end{align*}
	Using the assumptions of induction, we have
	\[
	\pas{-1}^{i-1}F_{k-i+2}+\alpha \pas{-1}^iF_{k-i+1}\equiv F_{k+i-2}+\alpha F_{k+i-1}=F_{k+i} \pmod{F_k}.
	\]
	The proof of $(\ref{lem:fractional part of Falpha-2})$ runs as
	\begin{align*}
		F_{2k+i}=F_{2k}F_{i+1}+F_{2k-1}F_i\equiv F_{2k-1}F_i\equiv \pas{-1}^kF_i \pmod{F_k}
	\end{align*}
	from the case of $i=k-1$ of $(\ref{lem:fractional part of Falpha-1})$. Similarly, $(\ref{lem:fractional part of Falpha-3})$ and $(\ref{lem:fractional part of Falpha-4})$ is obtained from
	\begin{align*}
		F_{3k+i}
		&=F_{3k}F_{i+1}+F_{3k-1}F_i\equiv F_{2k+(k-1)}F_i\equiv \pas{-1}^kF_{k-1}F_i\\
		&\equiv\pas{-1}^k\pas{F_kF_{i+1}+F_{k-1}F_i}=\pas{-1}^kF_{k+i}\equiv \pas{-1}^{k+i+1}F_{k-i} \pmod{F_k},\\
		F_{4k+i}
		&\equiv F_{4k-1}F_i\equiv\pas{-1}^{k+(k-1)+1}F_1F_i=F_i \pmod{F_k}.
	\end{align*}
	Next, we observe that $(\ref{lem:fractional part of Falpha-5})$ holds for every non-negative $a,b$ with $a\equiv b\pmod4$ if and only if for every non-negative $a$,
	\begin{align}\label{lem:fractional part of Falpha-6}
		F_{ak+i}\equiv F_{\delta k+i} \pmod{F_k}
	\end{align}
	holds where $a\equiv\delta\pmod4$ with $\delta\in\Pas{0,1,2,3}$. We prove $(\ref{lem:fractional part of Falpha-6})$ by using induction with respect to $a$. The case $a=0,1,2,$ and $3$ are trivial since $a=\delta$. And $(\ref{lem:fractional part of Falpha-4})$ is nothing but the case of $a=4$. Suppose that $(\ref{lem:fractional part of Falpha-6})$ holds also all natural numbers less than $a\geq5$, and we take $\delta^\prime\in\{1,2,3,4\}$ satisfying $\delta^\prime\equiv a\pmod4$. Then
	\begin{align*}
		F_{ak+i}=F_{ak}F_{i+1}+F_{ak-1}F_i\equiv F_{ak-1}F_i=F_{(a-1)k+(k-1)}F_i.
	\end{align*}
	Using the assumption of induction, we have
	\begin{align*}
		\equiv F_{(\delta^\prime-1)k+(k-1)}F_i=F_{\delta^\prime-1}F_i\equiv F_{\delta^\prime k}F_{i+1}+F_{\delta^\prime k-1}F_i=F_{\delta^\prime k+i} \pmod{F_k}.
	\end{align*}
\end{proof}

\begin{thm}\label{thm:best estimation of FDF(n)}
	Suppose $a\in\FDF{\N}$, and put $k=\ind{a}$. If $F_i\mid a$, then $i\leq(k+1)/2$, that is
	\[
	\ind{\FDF{a}}\leq\frac{k+1}{2}.
	\]
\end{thm}
\begin{proof}
	Let $i_0$ be the maximal $i$ satisfying $F_i\mid a$. The case $k=1$ is $i_0=1$ since $a=1$, and hence the claim holds. Assume that $i_0>(k+1)/2$ for $k\geq2$. In particular, $i_0\geq2$. Then $a>1$. Thus, there exist natural numbers $c_1\ldots,c_n$ such that
	\begin{align*}
		a=F_{i_0+c_1}+F_{i_0+c_2}+\cdots+F_{i_0+c_m}+F_{c_{m+1}}+\cdots+F_{c_n},\\
		k=i_0+c_1>i_0+c_2>\cdots>i_0+c_m>c_{m+1}>\cdots>c_n=1.
	\end{align*}
	Note that $c_1\leq i_0-2$. Since $F_{i_0}\mid a$,
	\begin{align*}
		a=\sum_{j=1}^mF_{i_0+c_j}+\sum_{j=m+1}^nF_{c_j}\equiv 0 \pmod{F_k}
	\end{align*}
	From Lemma.\ref{lem:fractional part of Falpha} $(\ref{lem:fractional part of Falpha-1})$,
	\begin{align*}
		\sum_{j=1}^mF_{i_0+c_j}\equiv \sum_{j=1}^m\pas{-1}^{c_j+1}F_{i_0-c_j} \pmod{F_k}
	\end{align*}
	It is enough to consider only the fractional part, and hence we suppose $c_{m+1}<i_0$ without loss of generality. Since $\alpha\geq3$ and $i_0\geq2$, we estimate
	\[
	\abs{\sum_{j=1}^m\pas{-1}^{c_j+1}F_{i_0-c_j}+\sum_{j=m+1}^nF_{c_j}}\leq2\sum_{j=1}^{i_0-1}F_j<F_{i_0}.
	\]
	Therefore, we have
	\[
	\sum_{j=m+1}^nF_{c_j}=\sum_{j=1}^m\pas{-1}^{c_j}F_{i_0-c_j}.
	\]
	Assume that there exist some $j\in\Pas{1,\ldots,m}$ such that $c_j$ is odd, and denote all of them by $c_{d_1},\ldots,c_{d_h}$. Then
	\[
	\sum_{j=m+1}^nF_{c_j}+\sum_{i=1}^hF_{i_0-c_{d_i}}=\sum_{\substack{j=1\\j\neq c_{d_i}(i=1,\ldots,h)}}^mF_{i_0-c_j}.
	\]
	In both sides, each coefficient of $\Falpha$ numbers are $\leq2$. Thus, they are Zeckendorf representations of a natural number. However, the left hand side has $F_{i_0-c_{d_h}}$ and the right-hand side does not. This is a contradiction to the uniqueness of the Zeckendorf representation. Therefore $c_j$ is even for all $j=1,\ldots,m$, and
	\[
	\sum_{j=m+1}^nF_{c_j}=\sum_{j=1}^mF_{i_0-c_j}
	\]
	holds. This both sides are Zeckendorf representations. But the left hand side has $F_1$ and the right-hand side does not have it since $c_1\leq i_0-2$. This is also a contradiction to the uniqueness, and hence we obtain that $i_0\leq(k+1)/2$.
\end{proof}

\begin{rem}\label{rem:best estimation of FDF(n)}
	The estimate given by Theorem.\ref{thm:best estimation of FDF(n)} is best-possible. Let $a=F_{2p-1}+1$ with $p,2p-1\in\P$. For example, $p=3$. From Theorem.\ref{thm:best estimation of FDF(n)}, $i\leq p$ if $F_i\mid a$. On the other hand, we get $F_p\mid a$ if $p$ is odd from Corollary.\ref{crl:divisors of F(2p+1)+1} $(\ref{crl:divisors of F(2p+1)+1-1})$.
\end{rem}

We can obtain that the converse of Corollary.\ref{crl:divisors of F(2p+1)+1} from this theorem.

\begin{thm}\label{thm:converse of cor.3.4}
	For two odd prime numbers $p,q$, the following are equivalent:
	\begin{enumerate}
		\item[\textup{(i)}] $\displaystyle p=2q+1\text{ or }2q-1$,
		\item[\textup{(ii)}] $\displaystyle \FDFk{2}{F_p}=\FDF{F_q} \text{ for some }\alpha\geq3$,
		\item[\textup{(iii)}] $\displaystyle \FDFk{2}{F_p}=\FDF{F_q} \text{ for all }\alpha\geq3$.
	\end{enumerate}
\end{thm}
\begin{proof}
	From Corollary.\ref{crl:divisors of F(2p+1)+1}, $(\text{i})$ implies $(\text{iii})$. Also it is clear that $(\text{ii})$ follows from $(\text{iii})$. Suppose that $(\text{ii})$ is true, and let $p=2q^\prime+1$. Since
	\begin{align*}
		F_p+1=F_{(q^\prime+1)+q^\prime}+F_{(q^\prime+1)-q^\prime}=
		\begin{cases}
			F_{q^\prime}\pas{F_{q^\prime+2}+F_{q^\prime}}	&	\pas{q^\prime\text{ is odd}},\\
			F_{q^\prime+1}\pas{F_{q^\prime+1}+F_{q^\prime-1}}	&	\pas{q^\prime\text{ is even}}
		\end{cases}
	\end{align*}
	from $(\ref{eq:summation formula-1})$,
	\begin{align*}
		\FDFk{2}{F_p}=\FDF{F_p+1}\geq
		\begin{cases}
			\FDF{F_{q^\prime}}	&	\pas{q^\prime \text{ is odd}},\\
			\FDF{F_{q^\prime+1}}	&	\pas{q^\prime \text{ is even}}.
		\end{cases}
	\end{align*}
	Thus
	\[
	q=\ind{F_q+1}=\ind{\FDF{F_q}}=\ind{\FDFk{2}{F_p}}\geq q^\prime+\frac{1+\pas{-1}^{q^\prime}}{2}.
	\]
	On the other hand,
	\[
	\ind{\FDFk{2}{F_p}}=\ind{\FDF{F_p+1}}\leq\frac{p+1}{2}=q^\prime+1
	\]
	from Theorem.\ref{thm:best estimation of FDF(n)}. If $q^\prime$ is even, then we have $q=q^\prime+1$, and hence $p=2q-1$. If $q^\prime$ is odd, then we get $q^\prime=q$ since $q^\prime\leq q\leq q^\prime+1$ and $q$ is odd. Consequently, we obtain that $p=2q+1$.
\end{proof}

\begin{thm}\label{thm:estimatation of ord(n) with ind(n)}
	For every natural number $n$,
	\[
	\ord{n}\leq \frac{1}{\log2}\log\pas{\ind{n}}+2.
	\]
\end{thm}
\begin{proof}
	Let $a$ be an arbitrary element in $\FDF{\N}$, and take a natural number $k$ such that $2^k\leq\ind{a}<2^{k+1}$. From Theorem.\ref{thm:best estimation of FDF(n)}, we estimate that
	\[
	\ind{\FDFk{k}{a}}\leq \frac{1}{2}\ind{\FDFk{k-1}{a}}+\frac{1}{2}\leq\cdots\leq\pas{\frac{1}{2}}^k\ind{a}+\sum_{m=1}^{k-1}\frac{1}{2^m}<3.
	\]
	That is $\ind{\FDFk{k}{a}}\leq2$. For $b\in\FDF{\N}$, if $\ind{b}=2$, then $b$ is $F_2+F_1$. Since $\alpha$ does not divide them, we get $\FDF{b}=1$. Therefore $\ind{\FDFk{k+1}{a}}=1$. Since $k\leq \log(\ind{a})/\log2$, we find that
	\[
	\ord{a}\leq\flog{2}\log\pas{\ind{a}}+1.
	\]
	From this, for every $n$,
	\[
	\ord{n}=\ord{\FDF{n}}+1\leq\flog{2}\log\pas{\ind{\FDF{n}}}+2\leq \flog{2}\log\pas{\ind{n}}+2.
	\]
\end{proof}

Let us remember that $\ind{n}$ is the maximal index of $\Falpha$ numbers $\leq n$. Since $\Falpha$ is defined by the linear recurrence sequence whose coefficients are $\geq1$, it diverges exponentically. Therefore, we expect that $\ord{n}\ll\log\log n$. In fact, we obtain the following more explicit inequalities.

\begin{thm}\label{thm:estimation of ord(n)}
	Let $\alpha\geq3$ be an integer. Then we have
	\begin{align*}
		\ord{n}
		\begin{cases}
			=0                       & (n=1),\\
			<\frac{1}{\log2}\log\log n+3 & (n\geq2).
		\end{cases}
	\end{align*}
	In particular, if $\alpha\geq55$ and
	\[
	n>\exp\pas{\frac{4}{\log\valpha-4}\log\pas{\sqrt{\alpha^2+4}+\frac{1}{\valpha}}},
	\]
	then we get
	\begin{align}\label{thm:estimation of ord(n)-1}
		\ord{n}<\frac{1}{\log 2}\log\log n.
	\end{align}
	Moreover,
	\begin{align*}
		\ord{n}
		\begin{cases}
			=0 & (n=1),\\
			=1 & (2\leq n\leq7),\\
			<\frac{1}{\log2}\log\log n & (n\geq8)
		\end{cases}
	\end{align*}
	holds at least in the case $\alpha\geq2981$.
\end{thm}
\begin{proof}
	We see that $\ord{1}=0$ by definition. Hereafter let $n\geq2$ and $k=\ind{n}\geq1$. Since $F_k\leq n$, we have
	\[
	k\leq\flog{\valpha}\log\pas{n\sqrt{\alpha^2+4}+\pas{-\valpha}^{-k}}
	\]
	by the argument in the proof of Theorem.\ref{thm:there exists a k s.t. FDFk(n)=1}. From Theorem.\ref{thm:estimatation of ord(n) with ind(n)}, we estimate that
	\begin{align}\label{thm:estimation of ord(n)-2}
		\ord{n}
		&\leq\flog{2}\pas{\log\log\pas{n\sqrt{\alpha^2+4}+\pas{-\valpha}^{-k}}-\log\log\valpha}+2\notag\\
		&<\flog{2}\log\log n+\flog{2}\pas{\log\pas{1+\flog{n}\log\pas{\sqrt{\alpha^2+4}+\frac{1}{\valpha}}}-\log\log\valpha}+2.
	\end{align}
	Here, we define
	\begin{align*}
		f_\alpha(x)&:=\log\pas{1+\flog{x}\log\pas{\sqrt{\alpha^2+4}+\frac{1}{\valpha}}}-\log\log\valpha,\\
		g_\alpha(x)&:=\exp\pas{f_\alpha(x)}=\flog{\valpha}\pas{1+\flog{x}\log\pas{\sqrt{\alpha^2+4}+\frac{1}{\valpha}}}
	\end{align*}
	with $x\geq2$. For real $y>(\log\valpha)^{-1}$, we have
	\[
	x>\exp\pas{\frac{1}{y\log\valpha-1}\log\pas{\sqrt{\alpha^2+4}+\frac{1}{\valpha}}}.
	\]
	Denote by $A_\alpha(y)$ the right-hand side of this inequality. Then $A_\alpha(y)$ is decreasing with respect to $\alpha$. (We will prove this in Remark.\ref{rem:estimation of ord(n)}.) Since $\varphi_3=(3+\sqrt{13})/2>3.3$, we have
	\[
	\log A_\alpha(2)\leq\log A_3(2)=\frac{1}{2\log\varphi_3-1}\log\pas{\sqrt{13}+\varphi_3^{-1}}<1
	\]
	with $\alpha\geq3$, that is $A_\alpha(2)<3$. Therefore, we have $g_\alpha(x)<2$ for $\alpha\geq3$ and $x\geq3$. Thus for every $\alpha\geq3$,
	\[
	\ord{n}<\flog{2}\log\log n+\flog{2}\log g_\alpha(n)+2<\flog{2}\log\log n+3
	\]
	with $n\geq3$. In addition, this also holds the case $n=2$ since $\ord{2}=1$ and $\log\log2/\log2+3\simeq2.47$.

	Let us next prove $(\ref{thm:estimation of ord(n)-1})$. It is enough to consider the domain of $\alpha\geq3$ and $a\geq2$ which satisfies $f_\alpha(x)/\log2+2<0$. Since this condition can be transformed into $g_\alpha(s)<1/4$ it is enough to assume the condition $x>A_\alpha(1/4)$ for all sufficient large $\alpha$. Then $\log\valpha>4$, that is $\alpha\geq55$. Thus for every $\alpha\geq55$, we have
	\[
	\ord{n}<\flog{2}\log\log n
	\]
	with $x>A_\alpha(1/4)$. Since $A_\alpha(y)$ is decreasing in $\alpha$, $\alpha>A_\alpha(1/4)$ holds for all sufficiently large $\alpha\geq55$. The lower bound is $\alpha\geq2981$ from computer calculations, and hence for every $\alpha\geq2981$ we get
	\[
	\ord{n}<\flog{2}\log\log n
	\]
	with $n\geq\alpha$. $\ord{n}=1$ if $2\leq n<\alpha$, and $\log\log n/\log2>1$ for $n\geq8$. This implies
	\begin{align*}
		\ord{n}
		\begin{cases}
			=1  & \pas{2\leq n\leq 7},\\
			<\flog{2}\log\log n  & \pas{n\geq8}.
		\end{cases}
	\end{align*}
\end{proof}

\begin{rem}\label{rem:estimation of ord(n)}
	We show that $A_\alpha(y)$ is decreasing monotonically in $\alpha$. By definition,
	\[
	\log A_\alpha(y)=\frac{\log\sqrt{\alpha^2+4}}{y\log\valpha-1}+\frac{1}{y\log\valpha-1}\log\pas{1+\frac{1}{\valpha\sqrt{\alpha^2+4}}}=:B(\alpha)+C(\alpha)
	\]
	say, with $y\log\valpha>1$. It is clear that $C(\alpha)$ is decreasing, and hence it is sufficient to discuss on $B(\alpha)$. For real $\alpha$ with $y\log\valpha>1$,
	\begin{align*}
		\frac{d}{d\alpha}B(\alpha)=\frac{1}{\pas{y\log\valpha-1}^2}\pas{\frac{2y\alpha}{\alpha^2+4}\log\valpha-\frac{2\alpha}{\alpha^2+4}-\frac{y\valpha^\prime}{\valpha}\log\pas{\alpha^2+4}}.
	\end{align*}
	Since $\valpha<\sqrt{\alpha^2+4}$, we estimate that the right-hand side of the above is
	\begin{align*}
		&<\frac{1}{\pas{y\log\valpha-1}^2}\pas{\frac{2y\alpha}{\alpha^2+4}\log\sqrt{\alpha^2+4}-\frac{y\valpha^\prime}{\valpha}\log\pas{\alpha^2+4}}\\
		&=\frac{y\log\pas{\alpha^2+4}}{\pas{y\log\valpha-1}^2}\pas{\frac{\alpha}{\alpha^2+4}-\frac{\valpha^\prime}{\valpha}}\\
		&=\frac{y\log\pas{\alpha^2+4}}{\pas{y\log\valpha-1}^2\pas{\alpha^2+4}\pas{\alpha+\sqrt{\alpha^2+4}}}\pas{\alpha\pas{\alpha+\sqrt{\alpha^2+4}}-\pas{1+\frac{\alpha}{\sqrt{\alpha^2+4}}}\pas{\alpha^2+4}}\\
		&=-\frac{4y\log\pas{\alpha^2+4}}{\pas{y\log\valpha-1}^2\pas{\alpha^2+4}\pas{\alpha+\sqrt{\alpha^2+4}}}<0.
	\end{align*}
\end{rem}

Let us consider the case $\alpha=55$.

\begin{rem}
	Since $\log A_{55}\pas{1/4}\simeq 2091.79$,
	\[
	\textup{ord}_{55}\pas{n}<\frac{1}{\log2}\log\log n
	\]
	holds at least for $n>e^{2092}$.
\end{rem}

In addition, we obtain the following corollary from Theorem.\ref{thm:estimation of ord(n)}.

\begin{crl}
	For every $\alpha\geq3$,
	\[
	\limsup_{n\to\infty}\frac{\ord{n}}{\log\log n}\leq \frac{1}{\log2}.
	\]
\end{crl}

\begin{crl}
	\begin{align*}
		\lim_{\alpha\to\infty}\liminf_{n\to\infty}\pas{\flog{2}\log\log n-\ord{n}}=+\infty.
	\end{align*}
\end{crl}
\begin{proof}
	From $(\ref{thm:estimation of ord(n)-2})$, we have
	\[
	\liminf_{n\to\infty}\pas{\flog{2}\log\log n-\ord{n}}\geq\flog{2}\log\log\valpha-2.
	\]
\end{proof}

From Theorem.\ref{thm:estimatation of ord(n) with ind(n)}, we have $\ord{F_p}\leq \log p/\log2 +2$ for prime $p$. In fact, there is quite a big difference between them, which can be observed by numerical tests $(\textup{FIGURE}\ref{FIGURE 1},\textup{FIGURE}\ref{FIGURE 2})$. Thus, the author believes that the following theorem will be useful in the future.

\begin{figure}[ht]
	\centering
	\includegraphics[width=0.6\columnwidth]{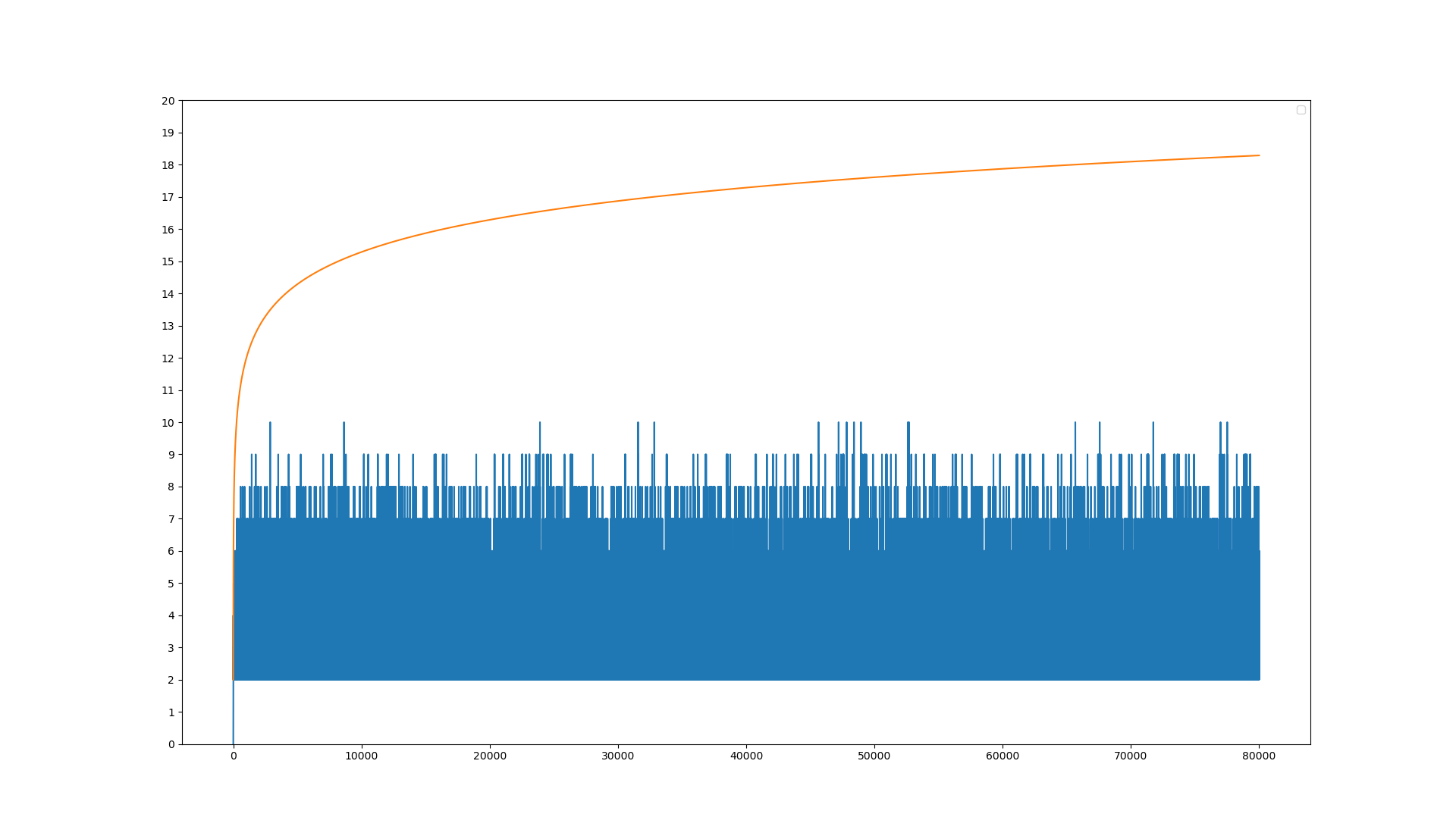}
	\caption{$\textup{ord}_3\pas{F_n}$ and $\log n/\log2+2\;(n\leq80000)$}
	\label{FIGURE 1}
\end{figure}
\begin{figure}[ht]
	\centering
	\includegraphics[width=0.6\columnwidth]{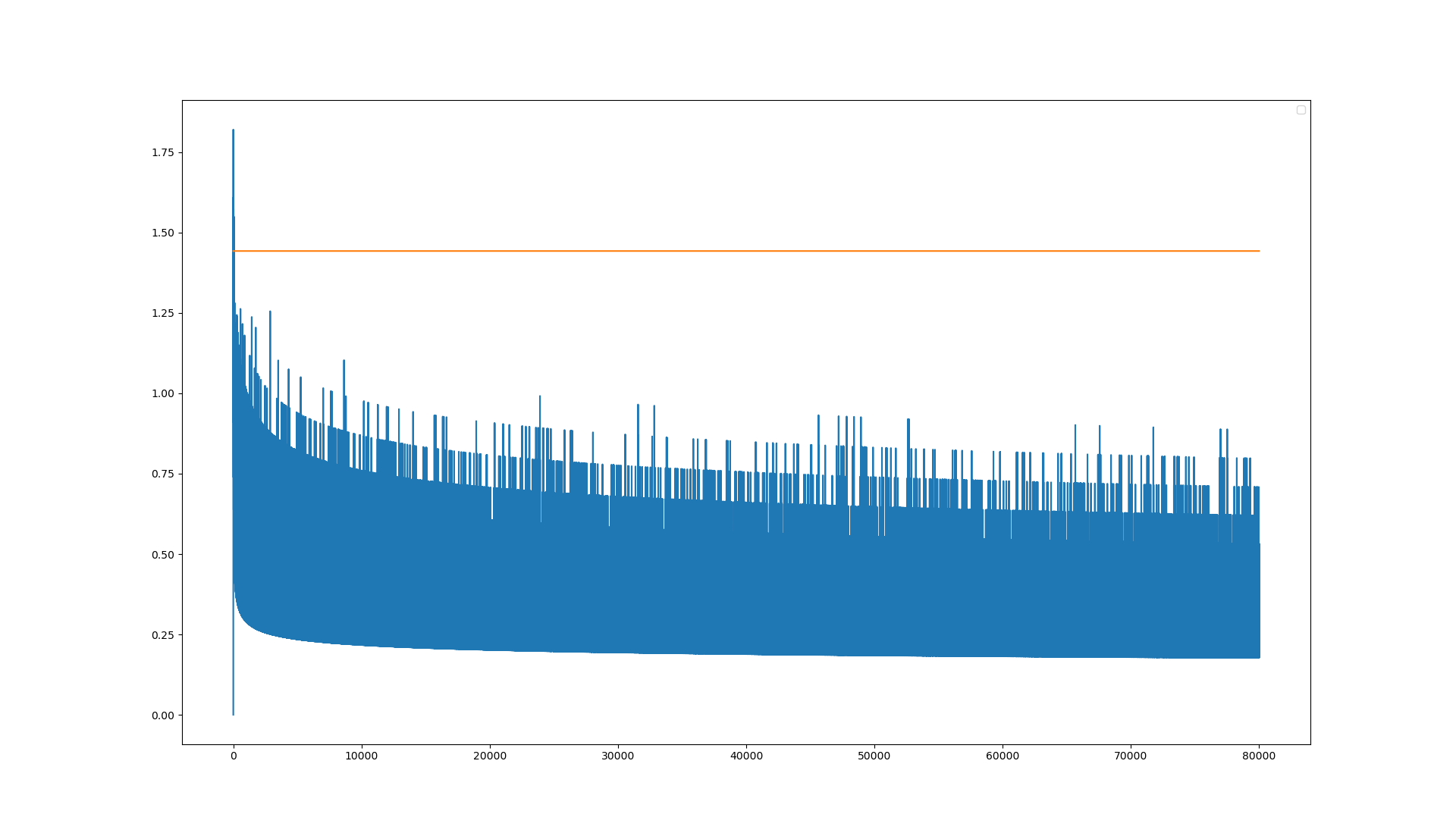}
	\caption{$\textup{ord}_3\pas{F_n}/\log (n+1)$ and $1/\log2\;(n\leq80000)$}
	\label{FIGURE 2}
\end{figure}

\newpage

\begin{thm}\label{thm:related with CC}
	Suppose $\alpha\geq3$, and put $C_\alpha:=\limsup_{p\to\infty}\frac{\ord{F_p}}{\log p}$. If $C_\alpha<1/\log2$ for some $\alpha$, then
	\[
	\limsup_{p\to\infty}\frac{l(p)}{\log p}\leq\frac{C_\alpha}{1-C_\alpha\log2}.
	\]
\end{thm}
\begin{proof}
	Suppose that $C_\alpha<1/\log2$. For every $0<\ep<1/\log2-C_\alpha$,
	\begin{align}\label{thm:related with CC-1}
		\ord{F_p}<\pas{C_\alpha+\ep}\log p
	\end{align}
	with sufficiently large $p$. And
	\[
	\log\pas{\pas{p\pm1}2^{l_{\pm1}(p)-1}\mp1}=\pas{l_{\pm1}(p)-1}\log2+\log p+O\pas{\frac{1}{p}}.
	\]
	We replace $p$ by $\pas{p\pm1}2^{l_{\pm1}(p)-1}\mp1$ in $(\ref{thm:related with CC-1})$. Then
	\[
	l(p)-1+\ord{F_p}<\pas{C_\alpha+\ep}\pas{\pas{l_{\pm1}(p)-1}\log2+\log p+O\pas{\frac{1}{p}}}
	\]
	from Corollary.\ref{crl:divisors of F(2p+1)+1}, and hence we have
	\[
	l(p)<1+\frac{C_\alpha+\ep}{1-(C_\alpha+\ep)\log2}\log p+O\pas{\frac{1}{p}}.
	\]
	Divide the both sides by $\log p$, and take limit superior with respect to $p$. Then we find that
	\[
	\limsup_{p\to\infty}\frac{l(p)}{\log p}\leq\frac{C_\alpha+\ep}{1-\pas{C_\alpha+\ep}\log2}.
	\]
\end{proof}

The sufficient condition of Theorem.\ref{thm:related with CC}, written in terms of prime numbers, can be replaced by the condition written in terms of natural numbers.

\begin{crl}\label{crl:related with CC}
	Suppose that $\alpha\geq3$, and put $D_\alpha:=\limsup_{n\to\infty}\frac{\ord{n}}{\log\log n}$. If $D_\alpha<1/\log2$ for some $\alpha$, then
	\[
	\limsup_{p\to\infty}\frac{l(p)}{\log p}\leq\frac{D_\alpha}{1-D_\alpha\log2}.
	\]
\end{crl}
\begin{proof}
	For all natural $n$,
	\[
	\log F_n=n\log\valpha+\log\pas{\frac{1-\pas{-\valpha^2}^{-n}}{\sqrt{\alpha^2+4}}}
	\]
	from Fact.\ref{fact:fact of Fa numbers} $(a)$. In particular, $\log\log F_n\sim \log n$. Then, for every $\ep>0$, we have $\log p/\log\log F_p>1-\ep$ with any sufficiently large $p$. Thus we estimate that
	\[
	D_\alpha\geq\limsup_{p\to\infty}\frac{\ord{F_p}}{\log\log F_p}=\limsup_{p\to\infty}\frac{\ord{F_p}}{\log p}\cdot\frac{\log p}{\log\log F_p}\geq\pas{1-\ep}\limsup_{p\to\infty}\frac{\ord{F_p}}{\log p}.
	\]
	Now, since $\ep>0$ is arbitrary, we get
	\[
	\limsup_{p\to\infty}\frac{\ord{F_p}}{\log p}\leq D_\alpha.
	\]
	Here we suppose $D_\alpha<1/\log2$ and apply Theorem.\ref{thm:related with CC}, then the proof is complete.
\end{proof}

The advantage of this corollary is that the problem of upper estimation of $l(p)$ is reduced to the situation that we can use number theoretic methods which cannot be applied to prime numbers.

\section{Remaining Problems}
In this paper, an experimentally reliable sufficient condition for $l(p)\ll\log p$ was obtained using elementary methods that do not involve differentiation and integration. If we could successfully use analytical methods, perhaps we would obtain better estimation. For example, it is describable to find some analogy of $(\ref{eq:Dirichlet product})$ with respect to $\fdf^2,\fdf^3,\cdots$, or some non-trivial order of $\sum\ord{n}$. However, the difficulty lies in the fact that $\ord{n}$ is defined by the iterations of $\fdf$. The iterations of the divisor function $\sigma(n)$ and the Euler function $\varphi(n)$ have been considered in \cite{Erdos},\cite{Erdos-Granville-Pomerance-Spiro},\cite{Erdos-Subbarao} and so on; however those researches seem to be possible because $\sigma,\varphi$ are number theoretically easier to treat. Even though $\fdf$ is not multiplicative, $\Falpha$ numbers has some nice properties related to multiplication, such as Fact.\ref{fact:fact of Fa numbers} $(e),(f)$.
In the future, it is expected that such properties will be used well to obtain further results on the function $\textup{ord}_\alpha$.

Finally, we list the problems to be solved.
\begin{problem}
	\[
	\limsup_{n\to\infty}\frac{\ord{n}}{\log\log n}\overset{?}{<}\flog{2}
	\]
\end{problem}
It is shown in Theorem.\ref{crl:related with CC} that $l(p)\ll\log p$ holds if this inequality is true. Here we recall Conjecture.\ref{cjt:omega order of Cunningham chains}, which is not so far from the above inequality.
\begin{problem}
	Is there a sequence that is different from $\Falpha$ with ``similar properties"?
\end{problem}
The term ``similar properties" means those that are related to a chain of prime numbers.


\begin{thebibliography}{99}
	\bibitem{Agrawal-Kayal-Saxena} M. Agrawal, N. Kayal, N. Saxena, \textit{PRIMES is in P}, Annals of Mathematics, vol.\textbf{160}(2) (2004), pp.781-793.

	\bibitem{Artin} E. Artin, \textit{The Collected Papers of Emil Artin}, Edited by Serge Lang and John T. Tate, Addison-Wesley, Inc., Reading, Mass.-London (1965).

	\bibitem{Augustin} D. Augustin, \textit{Cunningham Chain records}, http://primerecords.dk/Cunningham\_Chain\_records.htm

	\bibitem{Bateman-Horn} P. T. Bateman, R. A. Horn, \textit{A Heuristic Asymptotic Formula Concerning the Distribution of Prime Numbers}, Mathematics of Computation, vol.\textbf{16}(79) (1962), pp.363-367.

	\bibitem{Caldwell} C. K. Caldwell, \textit{An Amazing Prime Heuristic}, arXiv:2103.04483 (2021).

	\bibitem{Dickson} L. E. Dickson, \textit{A New Extension of Dirichlet`s Theorem on Prime Numbers}, The Messenger of Mathematics, vol.\textbf{133} (1903-1905), pp.155-161.

	\bibitem{Diffie-Hellman} W. Diffie, M. E. Hellman, \textit{New Directions in Cryptography}, IEEE Transactions on Information Theory, vol.\textbf{IT-22}(6) (1976), pp.644-654.

	\bibitem{Dirichlet} P. G. L. Dirichlet, \textit{Beweis des Satzes, dass jede unbegrenzte arithmetische Progression, deren erstes Glied und Differenz ganze Zahlen ohne gemeinschaftlichen Factor sind, unendlich viele Primzahlen enth\"{a}lt}, Abhandlungen der K\"{o}niglichen Akademie der Wissenschaften zu Berlin (1837), pp.45-81.

	\bibitem{Egami} S. Egami, \textit{Some curious Dirichlet series}, RIMS K\^{o}ky\^{u}roku, vol.\textbf{1091} (1999), pp.172-174.

	\bibitem{Erdos} P. Erd\"{o}s, \textit{Some remarks on the iterates of the $\varphi$ and $\sigma$ functions}, Colloquium Mathematicum, Institute of Mathematics Polish Academy of Sciences, vol.\textbf{17}(2) (1967), pp.192-202.

	\bibitem{Erdos-Granville-Pomerance-Spiro} P. Erd\"{o}s, A. Granville, C. Pomerance and C. Spiro, \textit{On the Normal Behavior of the Iterates Of some Arithmetic Functions}, Analytic Number Theory, Progress in Mathematics, vol.\textbf{85} (1990), pp.165-204.

	\bibitem{Erdos-Subbarao} P. Erd\"{o}s, M. V. Subbarao, \textit{On the iterates of some arithmetic functions}, The theory of arithmetic functions, Lecture Notes in Mathematics, vol.\textbf{251} (1972), pp.119-125.

	\bibitem{Green-Tao} B. Green and T. Tao, \textit{The primes contain arbitrarily long arithmetic progressions}, Annals of Mathematics, vol.\textbf{167} (2008), pp.481-547.

	\bibitem{Hoggatt} V. E. Hoggatt, Jr, \textit{Generalized Zeckendorf Theorem}, The Fibonacci Quarterly, vol.\textbf{10} (1972), pp.89-93.

	\bibitem{Keller} T. J. Keller, \textit{Generalizations of Zeckendorf`s Theorem}, The Fibonacci Quarterly, vol.\textbf{10} (1972), pp.95-102,112.

	\bibitem{Koshy} T. Koshy, \textit{Fibonacci and Lucas with Applications Second Edition}, Wiley, vol.\textbf{1} (2018).

	\bibitem{Lenstra} H. W. Lenstra Jr, \textit{Factoring integers with elliptic curves}, Annals of mathematics, vol.\textbf{126}(3) (1987), pp.649-673.

	\bibitem{Lenstra-Pomerance} H. W. Lenstra Jr, C. Pomerance, \textit{Primality testing with Gaussian periods}, European Mathematical Society, vol.\textbf{21}(4) (2019), pp.1229-1269.

	\bibitem{Navas} L. Navas, \textit{Analytic continuation of the Fibonacci Dirichlet series}. The Fibonacci Quarterly, vol.\textbf{39} (2000), pp.409-418.

	\bibitem{Pollard} J. M. Pollard, \textit{Theorems on factorization and primality testing}, Proceedings of the Cambridge Philosophical Society, vol.\textbf{76} (1974), pp.521–528.

	\bibitem{Rosser-Schoenfeld} J. B. Rosser, L. Schoenfeld, \textit{Approximate formulas for some functions of prime numbers}, Illinois Journal of Mathematics vol.\textbf{6}(1) (1962), pp.64-94.
\end{thebibliography}
\end{document}